\documentclass{article}

    \usepackage[utf8]{inputenc}
    \usepackage{verbatim}
    \usepackage{amsfonts}
    \usepackage{amsmath}
    \usepackage{amsthm}
    \usepackage{stmaryrd} 
    
    \usepackage[style=alphabetic,maxbibnames=99]{biblatex}
        \DeclareFieldFormat{title}{\mkbibquote{#1}}
    
    \usepackage{hyperref}
    \usepackage[noabbrev,nameinlink,capitalise]{cleveref}
        \crefname{subsection}{Subsection}{Subsections}
        \crefname{subsection}{Subsection}{Subsections}
    
    \usepackage{subcaption}
    \usepackage{color}
    \usepackage[dvipsnames]{xcolor}
    \usepackage{graphicx}

    \theoremstyle{plain}
        \newtheorem{theorem}{Theorem}[section]
        \newtheorem{corollary}[theorem]{Corollary}
        \newtheorem{lemma}[theorem]{Lemma}
        \newtheorem{proposition}[theorem]{Proposition}
    \theoremstyle{definition}
        
        \newtheorem{example}[theorem]{Example}
        
        \newtheorem{algorithm}[theorem]{Algorithm}
    \theoremstyle{remark}
        \newtheorem{remark}[theorem]{Remark}

    \title{Conjugacy and period in symmetric diagram groups}
    \author{Matteo Tarocchi}
    \date{}

    \newcommand{\init}[0]{\mathsf{i}}
    \newcommand{\term}[0]{\mathsf{t}}
    \DeclareMathOperator{\SD}{SD}
    \newcommand{\Presentation}[0]{\mathcal{C}}
    \newcommand{\Relations}[0]{\mathcal{R}}
    \newcommand{\Top}[0]{\mathrm{Top}}
    \newcommand{\Bottom}[0]{\mathrm{Bot}}
    \newcommand{\Transistors}[0]{\mathrm{Trn}}
    \newcommand{\labeling}[0]{\mathsf{L}}
    \newcommand{\coloring}[0]{\mathsf{C}}
    \DeclareMathOperator{\domain}{Dom}
    \DeclareMathOperator{\range}{Ran}
    
    \newcommand{\semigp}{\mathcal{L}}
    \newcommand{\LoopMap}{\Lambda}
    \newcommand{\x}{\textcolor{Blue}{x}}
    \newcommand{\y}{\textcolor{Red}{y}}

\begin{document}

\maketitle

\begin{abstract}
    We show how to solve the conjugacy (search) problem and the problem of determining whether a given element is periodic in symmetric diagram groups based on finite semigroup presentations.
\end{abstract}


\section{Introduction}

Symmetric diagram groups are to Thompson's group $V$ what (standard) diagram groups are to Thompson's group $F$.
Briefly introduced by Guba and Sapir in \cite{DiagramGroups} under the name of \textit{braided diagram groups}, a formal study of symmetric diagram groups was later started by Farley in \cite{PictureGroupsCAT}, there called \textit{picture groups}, and was continued in articles such as \cite{FH18,GenevoisEmbeddings,BGN23}.
\footnote{
The author decided to forgo the names \textit{picture groups} and \textit{braided diagram groups} and opted for \textit{symmetric diagram group} as in \cite{BGN23,GenevoisDiagramGroups}, as this choice expressly suggests that they generalize (standard) diagram groups and because no actual braids are involved.
}

The purpose of this short paper is to present a solution to some algorithmic problems in symmetric diagram groups, namely the \textbf{conjugacy problem} (deciding whether two elements are conjugate), the \textbf{conjugacy search problem} (finding an explicit conjugator) and the \textbf{period problem} (deciding whether an element is periodic).
We develop and use a variation of strand diagrams, initially used in \cite{BM14} for the conjugacy problem in Thompson's groups $F$, $T$ and $V$.
These have already been a useful tool to understand conjugacy in many groups \cite{BM14,Aroca2018TheCP,GofferLederle,BM23,RearrConj,ConjTFGoES}.
Note that previous applications of strand diagrams only concerned tree-like diagrams (where vertices were either \textit{splits} or \textit{merges});
no such distinction exists for symmetric diagram groups.

\begin{theorem}
    Given any finite semigroup presentation $\Presentation=\langle \Sigma \mid \Relations \rangle$ and any word $w \in \Sigma^*$, the conjugacy problem, the conjugacy search problem and the period problem are solvable in the symmetric diagram group $\SD_w(\Presentation)$.
\end{theorem}

The strategy for conjugacy directly generalizes that of \cite{ConjTFGoES}, where one computes (non-unique) reduced closed diagrams and separately compares the parts with vertices using graph cohomology and the parts without vertices using the word problem in abelian semigroups.
Two key additions to \cite{ConjTFGoES} lie in properly reinterpreting symmetric diagrams as graphs and in dealing with loop rewritings not as reductions but as generally non-terminating transformations.

\begin{remark}
    There exist semigroup presentations $\Presentation$ and words $w\in\Sigma^*$ such that the symmetric diagram group $\SD_w(\Presentation)$ has solvable conjugacy problem whereas the diagram group $\mathrm{D}_w(\Presentation)$ does not by \cite[Example 15.22]{DiagramGroups}.
    What makes conjugacy easier to determine in symmetric diagram groups is that distinguishing loop parts is equivalent not to the word problem in the semigroup with presentation $\Presentation$, but in the abelian semigroup with presentation $\Presentation$.
\end{remark}

The algorithm for the period problem also relies on strand diagrams and generalizes that of \cite[Section 5]{BBIP26} for Thompson's groups.
Note that this problem is unsolvable in certain relatives of Thompson groups such as $2V$ \cite{WP2V}.

Diagram groups act (not always faithfully) on $[0,1]$ via transition functions as described in \cite[Section 4]{DgmGpsOrder}.
This easily generalizes to actions of symmetric diagram groups on the Cantor space.
Even if we do not do this here, closed strand diagrams could be used to analyze the dynamics of such actions.


\section{Symmetric diagram groups}

Let us define symmetric diagram groups using graphs.

\subsection{Basic definitions, notations and conventions}
\label{sub:notation:convention}

Let us briefly set the notation and recall the notions that we use.

\subsubsection{Graphs}
\label{ssub:graphs}

A \textbf{graph} is a tuple $(V,E,\init,\term)$, where $V$ is a non-empty set (of \textit{vertices}), $E$ is a set (of \textit{edges}), $\init$ and $\term$ are maps $E \to V$ ($\init(e)$ and $\term(e)$ are the \textit{initial} and \textit{terminal} vertices of an edge $e$, respectively).
For closed symmetric diagrams we use a generalized version of graphs where some edges, which we call \textbf{loops}, can be devoid of initial and terminal vertices (examples at the bottom of \cref{fig:confluence}).

A \textbf{cycle} is a subgraph with set of edges $\{e_1,\dots,e_k\}$ and set of vertices $\bigcup_{1\leq i\leq k}\left(\term(e_i)\cup\init(e_i)\right)$ such that $\term(e_i)=\init(e_j)\iff j\equiv i+1$ mod $k$.

\subsubsection{Semigroups}

We denote by $\Sigma^*$ the free semigroup over the set $\Sigma$.
This is the set of finite non-empty words over the alphabet $\Sigma$, composed by juxtaposition.

For a semigroup presentation $\Presentation=\langle\Sigma\mid\Relations\rangle$, we denote the relations in $\Relations$ by $(x=y)$ with $x,y\in\Sigma^*$.
We write $(x,y)^{-1}=(y,x)$ and $\Relations^{\pm1}=\Relations\cup\Relations^{-1}$.

We will always assume that $(x=y)\in\Relations \implies (x=y)^{-1}\notin\Relations$.

\subsubsection{Rewriting systems}

A \textbf{rewriting system} is just a graph.
We use this term when the edges are understood as \textbf{moves} that transform the initial vertex into the terminal one.

A rewriting system is:
\textbf{terminating} if it features no infinite walk (i.e., edges $e_1,e_2,\dots$ such that $\term(e_i)=\init(e_i)$ for all $i\geq1$);
\textbf{locally confluent} if, for all moves $A\rightarrow B_1$ and $A\rightarrow B_2$, there exist $C$ and walks from $B_1$ to $C$ and from $B_2$ to $C$;
\textbf{confluent} if there exists a unique $R$ (called the \textbf{reduced} element) such that, for all moves $A\rightarrow B_1$ and $A\rightarrow B_2$, there exist walks from $B_1$ to $R$ and from $B_2$ to $R$.
These notions are linked by the following classical result from \cite{NewmanDiamond}.

\begin{lemma}\label{lem:Newman:diamond}
    A locally confluent and terminating rewriting system is confluent.
\end{lemma}

\subsection{Symmetric diagrams}
\label{sub:symmetric:diagrams}

Let us fix a finite semigroup presentation $\Presentation=\langle\Sigma\mid\Relations\rangle$.
A \textbf{symmetric diagram} over $\Presentation$ is the collection $\Delta$ of the following data:
\begin{itemize}
    \item a graph $(V,E,\init,\term)$ without cycles;
    \item disjoint sets $\Top$, $\Bottom$ and $\Transistors$ such that $V=\Top\sqcup\Transistors\sqcup\Bottom$;
    \item a map $\labeling\colon\Transistors\to\Relations^{\pm1}$ (the \textit{labeling map});
    \item a map $\coloring\colon E\to\Sigma$ (the \textit{coloring map});
    \item for all $v\in V$, total orders $<_v^\term$ and $<_v^\init$ on $\term^{-1}(v)$ and $\init^{-1}(v)$, respectively;
    \item total orders $<^\Top$ and $<^\Bottom$ of $\Top$ and $\Bottom$, respectively;
\end{itemize}
satisfying the following conditions:
\begin{enumerate}
    \item all $v\in\Top$ are such that $|\term^{-1}(v)|=0$ and $|\init^{-1}(e)|=1$, while all vertices $v\in\Bottom$ are such that $|\term^{-1}(v)|=1$ and $|\init^{-1}(v)|=0$;
    \item for all $v\in\Transistors$, if $\term^{-1}(v)=\{e_1<_v^\term\dots<_v^\term e_n\}$ and $\init^{-1}(v)=\{f_1<_v^\init\dots <_v^\init f_m\}$ and if $\labeling(v)=(x_1x_2\dots x_{n'} = y_1y_2\dots y_{m'})$, then $n=n'$, $m=m'$, $\forall i\leq n,\ \coloring(e_i)=x_i$ and $\forall j\leq m,\ \coloring(f_j)=y_j$.
\end{enumerate}

The edges of a diagram are called \textbf{wires} and the elements of $\Top$, $\Bottom$ and $\Transistors$ are referred to as \textbf{sources}, \textbf{sinks} and \textbf{transistors}.

In pictures (such as \cref{fig:symmetric:diagram}), the total order $<_v^\term$ (resp. $<_v^\init$) is represented by the left-to-right order in which the elements of $\term^{-1}(v)$ (resp. $\init^{-1}(v)$) are drawn close to $v$.
Similarly, the total orders  $<^\Top$ and $<^\Bottom$ are represented by the left-to-right ordering of the vertices in $\Top$ and $\Bottom$, respectively.
Moreover, $\labeling(v)$ will be depicted as a symbol inside a circular node representing the transistor and $\coloring(e)$ will be represented by a color of the wire $e$.

\begin{example}\label{ex:symmetric:diagram:group}
    Our examples will be based on the semigroup presentation:
    \[\Presentation=\left\langle \x,\y\mid A=(\x,\x\y),B=(\x\y,\x\y\y\x),C=(\x\y,\x\x\y)\right\rangle.\]
    \cref{fig:symmetric:diagram} depicts two diagrams over $\Presentation$.
\end{example}

\begin{figure}
    \centering
    \includegraphics[scale=.75]{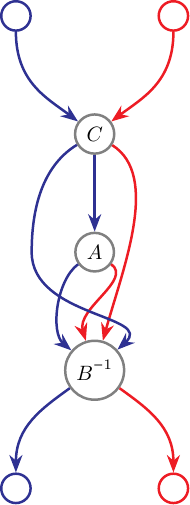}
    \hspace{4cm}%
    \includegraphics[scale=.75]{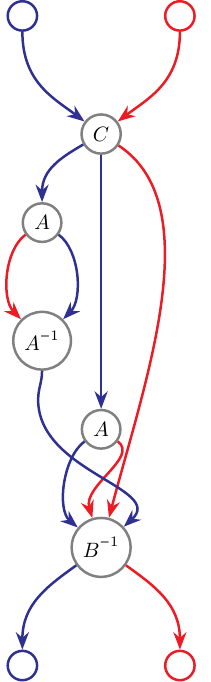}%
    \caption{Two symmetric diagrams over the presentation in \cref{ex:symmetric:diagram:group}.}
    \label{fig:symmetric:diagram}
\end{figure}

Among symmetric diagram groups one finds many subgroups of Thompson's group $V$, such as the Higman-Thompson groups $V_{n,r}$ and the Houghton groups $H_n$ (see \cite[Subsection 2.4]{GenevoisEmbeddings} and \cite{FH18}).
More generally, topological full groups of edge shifts are symmetric diagram groups (semigroups are generated by the edges of the graph and relations are $(e,w_e)$, with $w_e$ a word with a digit $f$ whenever $\term(e)=\init(f)$).
All these groups have solvable conjugacy problem by previous works, so we chose a recurring example that is less ``Thompsonesque''.


\subsection{Dipoles and equivalent diagrams}

A \textbf{dipole} is a subgraph $D(u,v)$ of a symmetric diagram over $\Presentation$ consisting of two transistors $u$ and $v$ and edges $e_1,\dots,e_k$ such that:
\begin{enumerate}
    \item $\labeling(u)=\labeling(v)^{-1}$;
    \item $\init^{-1}(u)=\{e_1<_u^\init e_2<_u^\init\dots<_u^\init e_k\}$;
    \item $\term^{-1}(v)=\{e_1<_v^\term e_2<_v^\term\dots<_v^\term e_k\}$.
\end{enumerate}

A \textbf{dipole reduction} consists of removing a dipole $D(u,v)$, merging the edges in $\term^{-1}(u)$ with those in $\init^{-1}(v)$ according to their orders $<_u^\term$ and $<_v^\init$.
Note that such edges and their colors are the same because $\labeling(u)=\labeling(v)^{-1}$.
A diagram is \textbf{reduced} if it features no dipoles.
For example, in \cref{fig:symmetric:diagram} a dipole reduction turns the right diagram into the left one, which is reduced.

We define a rewriting system on the family of symmetric diagrams over $\Presentation$ whose moves are the dipole reductions.
Two diagrams are \textbf{equivalent} (as symmetric diagrams) when they lie in the same connected component of the rewriting system.
A straightforward application of \cref{lem:Newman:diamond} shows that each diagram is equivalent to a unique reduced diagram, see \cite[Theorem 3.17]{DiagramGroups}.

\subsection{Symmetric diagram groupoids and groups}

Let $\Delta$ be a diagram with $\Top=\{t_1<^\Top\dots<^\Top t_n\}$, $\Bottom=\{b_1<^\Bottom\dots<^\Bottom b_m\}$, $x_i=\coloring(t_i)$ and $y_j=\coloring(b_j)$.
The \textbf{domain} and \textbf{range} of $\Delta$ are $\domain(\Delta)=x_1x_2\dots x_n\in\Sigma^*$ and $\range(\Delta)=y_1y_2\dots y_m\in\Sigma^*$, respectively.
For example, the diagrams in \cref{fig:symmetric:diagram} have domain and range equal to $\x\y$.

If $\range(\Delta_1)=\domain(\Delta_2)$, their composition $\Delta_1\Delta_2$ is computed by identifying $\Bottom_1$ with $\Top_2$ and merging the edges in $\term^{-1}(\Bottom_1)$ with those in $\init^{-1}(\Top_2)$ according to their orders $<^{\Bottom_1}$ and $<^{\Top_2}$.
This operation is well-defined on equivalence classes, i.e., if $\Delta_1$ and $\Delta_2$ are equivalent to $\Delta_1'$ and $\Delta_2'$, respectively, then $\Delta_1\Delta_2$ is equivalent to $\Delta_1'\Delta_2'$ and has the same domain and range.

This partial composition is associative.
The identity of a domain $x_1\dots x_n\in\Sigma^*$ is the equivalence class of the diagram with vertices $\Top=\{t_1<^\Top\dots<^\Top t_n\}$, $\Bottom=\{b_1<^\Bottom\dots<^\Bottom b_m\}$ and wires $e_1,\dots,e_n$ with $\init(e_i)=t_i$, $\term(e_i)=b_i$ and $\coloring(e_i)=x_i$.
The inverse is obtained by reversing the orientation of the edges and switching $\Top$ and $\Bottom$.
The set of all equivalence classes of symmetric diagrams over $\Presentation$ is thus a groupoid, which we denote by $\SD(\Presentation)$.
Given $w\in\Sigma^*$, a \textbf{symmetric diagram group} over $\Presentation$ with baseword $w$ is the group of $\Delta\in\SD(\Presentation)$ with $\domain(\Delta)=\range(\Delta)=w$.
We denote it by $\SD_w(\Presentation)$.

\begin{remark}\label{rmk:groupoid:generators}

    The groupoid $\SD(\Presentation)$ is generated by the infinite set of all \textbf{permutation diagrams} (diagrams with no transistors) and \textbf{simple transistor diagrams} (diagrams with a unique transistor and no crossings of wires in the planar drawing of the graph).
\end{remark}


\section{Closed symmetric diagrams}
\label{sec:closed:symmetric:diagrams}


If $\Delta$ is a symmetric diagram with $\domain(\Delta)=\range(\Delta)$, its \textbf{closed symmetric diagram} is the graph $\overline{\Delta}$ obtained by removing $\Top$ and $\Bottom$, merging $\init^{-1}(\Top)$ and $\term^{-1}(\Bottom)$ according to $<^{\Top}$ and $<^{\Bottom}$, and equipped with a totally ordered \textbf{base line}, which is the multiset of wires merged given the order inherited from $<^{\Top}$ and $<^{\Bottom}$.
(It is a multiset because wires may appear more than once.)
For example, \cref{fig:closed:symmetric:diagrams,fig:confluence} depict many close symmetric diagrams.

\begin{figure}
    \centering
    \includegraphics[width=\textwidth]{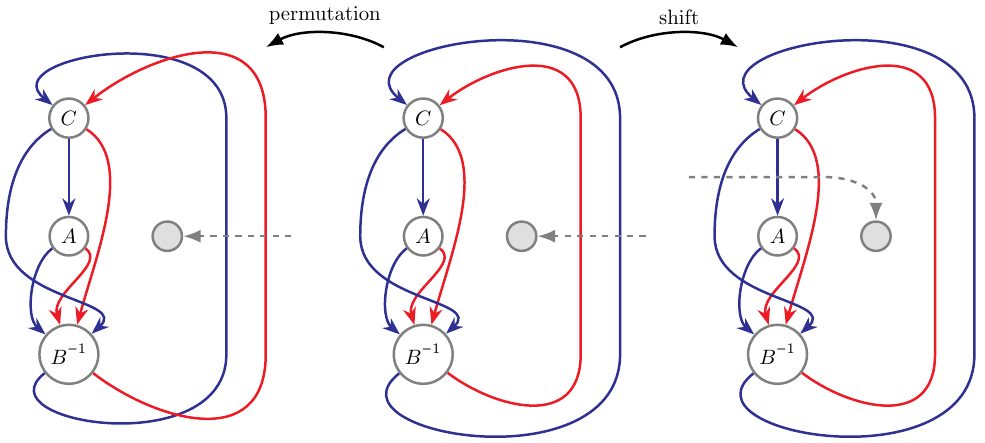}%
    \caption{The closed symmetric diagram associated to the leftmost diagram of \cref{fig:symmetric:diagram} along with a base permutation and a base line shift.}
    \label{fig:closed:symmetric:diagrams}
\end{figure}

The resulting graph has all the transistors of $\Delta$ but none of its sources and sinks.
The data given by the labeling $\labeling$, the coloring $\coloring$ and for all $v\in\Transistors$ the total orders $<_v^\init$ and $<_v^\term$ on $\term^{-1}(v)$ and $\init^{-1}(v)$, respectively, are all inherited from $\Delta$.
Note that $\overline{\Delta}$ also inherits condition 2 from \cref{sub:symmetric:diagrams}.

Observe that the result may not be a graph in the classical sense because of loops;
see \cref{ssub:graphs}.
Also note that these new graphs do not really encode any new information (i.e., $\overline{\Delta_1}=\overline{\Delta_2}\iff\Delta_1=\Delta_2$), but they allow us to more easily define certain transformations in the next subsections.

\subsection{Transformations of closed symmetric diagrams}
\label{sub:moves:closed:symmetric:diagrams}

Dipoles and dipole reductions of $\overline{\Delta}$ are precisely the same as those of $\Delta$.
Each dipole reduction of $\overline{\Delta}$ results in some $\overline{\Delta'}$ with $\Delta'$ equivalent to $\Delta$.

A \textbf{base permutation} is the transformation of applying a permutation $\sigma$ to the wires in the base line of $\overline{\Delta}$.
The result is $\overline{\Delta'}$ with $\Delta'=P^{-1}\Delta P$, where $P$ is the permutation diagram corresponding to $\sigma$.
See \cref{fig:closed:symmetric:diagrams}.

Let $\{e_1<\dots<e_r\}$ be the base line of $\overline{\Delta}$.
Assume that $\exists k,p,n\in\mathbb{N}$ such that $(e_{k+1},\dots,e_{k+n}),(e_{k+n+1},\dots,e_{k+2n}),\dots,(e_{k+(p-1)n+1},\dots,e_{k+pn})$ are the same tuples of loops (i.e., $e_{k+ln+i}=e_{k+ln+j}$ for all $l\in\{0,\dots,p-1\}$ and $i,j\in\{1,\dots,n\}$ congruent modulo $p$) and that these are all the instances of these edges in the base line.
Assume also that $\exists(x_1\dots x_n,y_1\dots y_m)\in\Relations^{\pm1}$ such that $\coloring(e_{k+ln+1})\coloring(e_{k+ln+2})\dots\coloring(e_{k+(l+1)n})=x_1\dots x_n$ for any (and then all) $l\in\{1,\dots,p\}$.
A \textbf{loop rewriting} is the transformation of replacing the loops $e_{k+1},\dots,e_{k+pn}$ with loops $f_1,\dots,f_{pm}$, in this order in the base line, where $f_{lm+i}$ is colored by $y_i$ for all $l\in\{1,\dots,p\}$ and $i\in\{1\dots,n\}$.
The result is $\overline{\Delta'}$ with $\Delta'=T^{-1}\Delta T$, where $T$ is a composition of $p$ transistor diagrams where the transistors are all labeled by $(x_1\dots x_n,y_1\dots y_m)$.
For example, \cref{fig:confluence} shows a loop rewriting at the bottom, where $p=2$.

Let $\{e_1<\dots<e_r\}$ be the base line of $\overline{\Delta}$.
Assume that $v$ is a transistor such that $\term^{-1}(v)$ (or $\init^{-1}(v)$) is $\{e_k,\dots,e_{k+n}\}$ for some $k\in\mathbb{N}$.
A \textbf{base line shifts} is the transformation of replacing $e_k<\dots<e_{k+n}$ in the base line with $\init^{-1}(v)$ (resp. $\term^{-1}(v)$) according to their order $<_v^\init$ (resp. $<_v^\term$).
This results in $\overline{\Delta'}$ with $\Delta'=T^{-1}\Delta T$, where $T$ is a transistor diagram whose sole transistor is labeled by $\labeling(v)$ (resp. $\labeling(v)^{-1}$).
For example, \cref{fig:closed:symmetric:diagrams} shows a base line shift.

\subsection{Rewriting with dipole reductions and similarities}

Together, base line shifts, loop rewritings and base permutations are referred to as \textbf{similarities}.
Note that, contrary to dipole reductions, similarities are reversible, i.e., each similarity $\overline{A}\rightarrow\overline{B}$ has an inverse similarity $\overline{B}\rightarrow\overline{A}$.

We define a rewriting system on the set of closed symmetric diagrams over $\SD(\Presentation)$ whose moves are dipole reductions and similarities.
Two closed symmetric diagrams are \textbf{equivalent} (as closed symmetric diagrams) if they lie in the same connected component of this rewriting system.

\begin{remark}\label{rmk:equivalent:to:similar}
    If $\Delta_1$ and $\Delta_2$ are equivalent as symmetric diagrams, then $\overline{\Delta_1}$ and $\overline{\Delta_2}$ are equivalent as closed symmetric diagrams.
\end{remark}

\begin{remark}\label{rmk:closed:equivalence:is:conjugacy}
    As noted while defining the transformations in \cref{sub:moves:closed:symmetric:diagrams}, similarities correspond to conjugacy by generators of $\SD(\Presentation)$.
    Therefore, if $\overline{A}$ and $\overline{B}$ are equivalent then $A$ and $B$ are conjugate in $\SD(\Presentation)$.
\end{remark}

\subsection{Strongly reduced diagrams and their decomposition}
\label{sub:strongly:reduced:and:parts}

We say that a closed symmetric diagram $\overline{\Delta}$ is \textbf{strongly reduced} if all representatives of its equivalence class have no dipoles.
Note that this is equivalent to having no dipole up to base line shifts and permutations, as loop rewritings have nothing to do with dipoles.
Whether $\overline{\Delta}$ is strongly reduced can thus be easily checked by forgetting about the base line and looking for dipoles.

A strongly reduced closed symmetric diagram has two parts:
its \textbf{loop part} is the subgraph consisting of all loops;
its \textbf{transistor part} is the subgraph consisting of everything else (or, equivalently, the union of those connected components which feature at least one transistor).

The \textbf{winding number} of a loop $l$ is the number of time $l$ appears in the base line (for example, the loops in \cref{fig:confluence} have winding number $2$).
Note that $l\geq1$ and that loop rewritings leave such numbers unaltered.


\section{The conjugacy problem}
\label{sec:conjugacy:problem}

\subsection{Closed symmetric diagrams and conjugacy}

The ties between conjugacy and equivalence of closed symmetric diagrams lie in the following fact.

\begin{proposition}
\label{prop:conjugacy:rewriting:system}
    Two symmetric diagrams $A$ and $B$ are conjugate in $\SD(\Presentation)$ if and only if $\overline{A}$ and $\overline{B}$ are equivalent closed symmetric diagrams.
\end{proposition}

\begin{proof}
    One direction is \cref{rmk:closed:equivalence:is:conjugacy}.
    For the other, assume that $A=X^{-1}BX$ in $\SD(\Presentation)$.
    By \cref{rmk:equivalent:to:similar}, the choice of representatives for $A$ and $B$ does not matter, so let $A$ be obtained by stacking the diagrams $X^{-1}$, $B$ and $X$ without performing dipole reductions.
    Close such diagram and perform base line shifts and base permutations until the copy of $X^{-1}$ (or $X$) has passed through the base line.
    Then $X^{-1}$ and $X$ cancel out, producing the diagram $B$.
\end{proof}

\subsection{Confluence up to similarities}

Recall the notion of \textit{strongly reduced} diagrams from \cref{sub:strongly:reduced:and:parts}.

\begin{proposition}
\label{prop:dipole:reductions}
    If $\overline{A}$ and $\overline{B}$ are equivalent and strongly reduced, then they are similar.
\end{proposition}

\begin{proof}
    Let us write $\llbracket\overline{A}\rrbracket$ to denote the similarity class of $\overline{A}$.
    Consider the rewriting system on the set of all similarity classes whose moves are dipole reductions on some representative of the similarity class.
    Note that the only possible obstacle to being able to reduce a dipole $D(u,v)$ is the presence of the base line between the two transistors $u$ and $v$, so if the same dipole can be reduced in $\overline{A_1},\overline{A_2}\in\llbracket\overline{A}\rrbracket$ then the resulting diagrams are in the same similarity class.

    Let us show that the rewriting system is confluent using \cref{lem:Newman:diamond}.
    That it is terminating is clear: each move strictly reduces the number of transistors, which is invariant in each similarity class.
    As for local confluence, let us assume that $\llbracket\overline{A}\rrbracket\rightarrow\llbracket\overline{B_1}\rrbracket$ and $\llbracket\overline{A}\rrbracket\rightarrow\llbracket\overline{B_2}\rrbracket$ are two distinct dipole reductions, say that they reduce dipoles $D(u_1,v_1)$ and $D(u_2,v_2)$, respectively.
    If $\{u_1,v_1\}\cap\{u_2,v_2\}=\emptyset$ then, up to possibly moving the base line out of the way, they can be performed in any order and so there is nothing to prove.
    If $u_1=u_2$ and $v_1=v_2$ then the two reductions are the same.
    If $u_1=v_2$ and $v_1\neq u_2$ or if $v_1=u_2$ and $u_1\neq v_2$, then the result of the two reductions is the same, exactly as happens in non-closed symmetric diagrams.
    If $u_1=v_2$ then we are in the situation described in \cref{fig:confluence}:
    a loop rewriting produces the same result and we are done.
\end{proof}

\begin{figure}
    \centering
    \includegraphics[width=\textwidth]{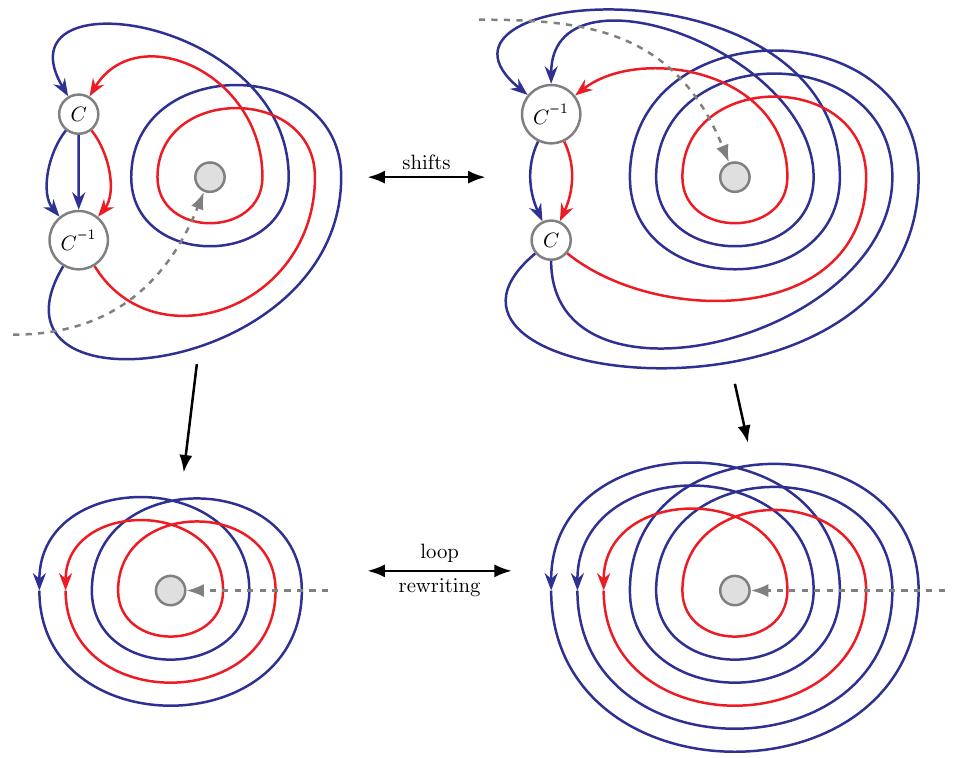}
    \caption{Local confluence up to similarity.}
    \label{fig:confluence}
\end{figure}

\subsection{Checking similarity}

For transistor parts, similarities are base line shifts and base permutations (there are no loop rewritings).
This is dealt with 
as done for Thompson’s group $V$ in
previous literature \cite{BM14,Nalecz}, so we omit all the details.

\begin{remark}\label{rmk:cohomology}
    Up to base permutations (which can be checked by brute force in finite time), the base line encodes graph cohomology, as explained in Rachel Nalecz's thesis \cite{Nalecz} or in \cite[Section 2]{BM14} about Thompson's group $V$.
    
    The base line corresponds to a coboundary in the graph underlying a closed symmetric diagram.
    Such coboundary is obtained by assigning to each edge $e$ the (non-negative) number of times $e$ appears in the base line.
    Two coboundaries are cohomologous precisely when the respective base lines differ by base line shifts and base permutations.
    In brief, that is because each base line shift prompts a sum or subtraction of the coboundary of a vertex and, conversely, two base lines with cohomologous coboundaries cut the diagrams in components that are valid symmetric diagrams (because they correspond to the same ``puncture'' in the sense of \cite{Nalecz}), and shifting through the components between the two base lines will move one base line to the other (see for example \cref{fig:cohomology}).
    The diagram to shift through is determined in more detail in the proof of Proposition 3.3 of \cite{BM14}, where it is denoted by $\mathfrak{h}_0$ (and an isomorphic copy $\mathfrak{h}_1$, one from the top and the other from the top of the infinite ``universal cover'' diagram).
    
    The thesis \cite{Nalecz} shows how to decide whether two coboundaries are cohomologous for Thompson's group $V$.
    Straightforward modifications produce an algorithm that determines whether two base lines differ from similarities.
\end{remark}

\begin{figure}
    \centering
    \includegraphics[scale=.75]{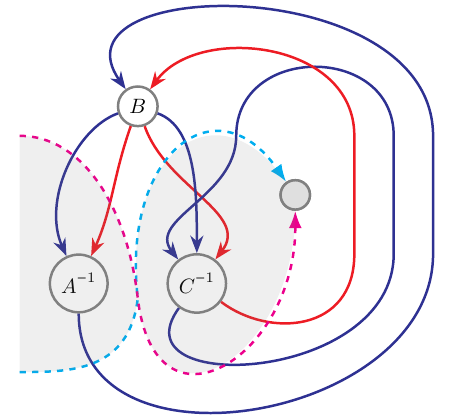}%
    \caption{Cohomologous coboundaries and areas to be shifted through.}
    \label{fig:cohomology}
\end{figure}

For loop parts, we need to decide when they are the same up to loop rewritings and base permutations.
For each $N\in\mathbb{N}\cup\{\infty\}$, let $\semigp_N$ denote the abelian semigroup generated by $x^{(i)}$ for all $x\in\Sigma$ and $1\leq i\leq N$, subject to relations
\[x_1^{(i)}x_2^{(i)}\dots x_n^{(i)}=y_1^{(i)}y_2^{(i)}\dots y_m^{(i)}\]
for all $1\leq i\leq N$ and all relation $(x_1x_2\dots x_n,y_1y_2\dots y_N)\in\Presentation$.
Note that $\semigp_N$ is a finitely presented abelian semigroup as soon as $\Presentation$ is a finite semigroup presentation and $N\neq\infty$.
Furthermore, $\semigp_1<\semigp_2<\dots<\semigp_\infty$.

Let $\LoopMap$ be the map that associates to each loop colored by $x\in\Sigma$ and with winding number $k\geq1$ the element $x^{(i)}$ of $\semigp_k$.
This extends to a map that sends the loop part of any strongly reduced closed symmetric diagram with loops $l_1,\dots,l_n$ to the sum $\LoopMap(l_1)+\dots+\LoopMap(l_n)\in\semigp_N$, for $N$ any number greater than or equal to the winding number of each $l_i$.

\begin{proposition}\label{prop:loop:semigroup}
    Loop parts $A$ and $B$ are similar $\iff\LoopMap(A)=_{\semigp_\infty}\LoopMap(B)$.
\end{proposition}

\begin{proof}
    Observe that similarities on the loop parts of closed symmetric diagrams are only loop rewritings and base permutations.
    Moreover, if two loop parts $A_1$ and $A_2$ are the same up to a base permutation, then $\LoopMap(A_1)=\LoopMap(A_2)$, since the semigroups $\semigp_N$ are abelian.
    Let $A$ be a loop part and assume that a loop rewriting $A\rightarrow B$ replaces the loops $a_1,\dots,a_n$ with loops $b_1\dots,b_m$.
    Let $k$ be the winding number of such loops and let $x_i=\coloring(a_i)$ (for $i=1,\dots,n$) and $y_j=\coloring(b_j)$ (for $j=1,\dots,m$).
    Then the abelian semigroup $\semigp_\infty$ is subject to the relation $x_1^{(k)}x_2^{(k)}\dots x_n^{(k)}=y_1^{(k)}y_2^{(k)}\dots y_m^{(k)}$.
    Therefore, loop rewritings (up to base permutations) corredpond one-to-one to applications of relations of $\semigp_\infty$, so the statement is proved.
\end{proof}

Each element of $\semigp_\infty$ belongs to $\semigp_N$ for an easily computable $N\in\mathbb{N}$.
Hence, two elements $a$, $b$ of $\semigp_\infty$ are the same in $\semigp_\infty$ if and only if they are the same in $\semigp_N$ for all $N$ such that $a,b\in\semigp_N$.
Finitely presented abelian semigroups have uniformly solvable word problem (see for example \cite[Theorem~5.1]{AbSemigroup}), so there is an algorithm that decides whether any two elements of $\semigp_\infty$ are the same.
Thanks to \cref{prop:loop:semigroup}, this proves the following fact.

\begin{corollary}\label{cor:loop:word:problem}
    The similarity of loop parts is decidable.
\end{corollary}

\subsection{Solving the conjugacy (search) problem}

Let $\Presentation$ be a fixed finite semigroup presentation and $w$ be a fixed word in the alphabet $\Sigma$.
The following algorithm checks whether two symmetric diagrams in $\SD(\Presentation)$ produce equivalent closed symmetric diagrams.

\begin{algorithm}
    Let $\Delta_1$, $\Delta_2$ be symmetric diagrams over $\Presentation$ with baseword $w$.
    \begin{enumerate}
        \item Compute $\overline{\Delta_1}$ and $\overline{\Delta_2}$ and reduce them up to similarities, performing base line shifts and base permutations whenever needed to unlock a dipole reduction.
        The result is unique up to similarity by \cref{prop:dipole:reductions}.
        \item Check whether the loop parts are similar by \cref{cor:loop:word:problem}.
        \item Check whether the transistor parts are similar by \cref{rmk:cohomology}.
    \end{enumerate}
\end{algorithm}

By \cref{prop:conjugacy:rewriting:system}, this is the same as checking conjugacy in the groupoid $\SD(\Presentation)$.
If two elements $g,h$ of a symmetric diagram group $\SD_w(\Presentation)$ are conjugate in the groupoid $\SD(\Presentation)$, then any conjugator must have domain and range equal to $w$ and must thus belong to $\SD_w(\Presentation)$, so being conjugate in $\SD(\Presentation)$ and $\SD_w(\Presentation)$ is the same for elements of $\SD_w(\Presentation)$.
Therefore, the previous algorithm decides conjugacy in $\SD_w(\Presentation)$, as desired.

The conjugacy search problem is solved by the fact that each similarity performed during a run of the algorithm builds pieces of a conjugator as a product of generators of $\SD(\Presentation)$, as noted throughout \cref{sub:moves:closed:symmetric:diagrams}.


\section{The period problem and computing periods}
\label{sec:period}

In this section we adapt and generalize the algorithm given in \cite[Section 5]{BBIP26} for Thompson's group $V$, arguably simplifying it along the way.

Closed diagrams encode information about conjugacy in the groupoid $\SD(\Presentation)$, which is enough for the purposes of determining periodicity and period:

\begin{remark}\label{rmk:period:conjugate}
    Let $A\in\SD_w(\Presentation)$ and $B\in\SD_v(\Presentation)$ be conjugate in $\SD(\Presentation)$.
    Then $A$ is periodic if and only if $B$ is, in which case they have the same period.
\end{remark}

When finite, periods can be easily computed thanks to the following fact.

\begin{proposition}\label{prop:period}
    If that $\Delta$ has no transistors, then $\Delta$ is periodic and its period equals the least common multiple of the winding numbers of the loops of $\overline{\Delta}$.
\end{proposition}

\begin{proof}
    The winding number of a loop determines the minimum number of times the loop passes through the base line, where each pass corresponds to moving through a copy of $\Delta$.
    Therefore, if a loop has winding number $k$ then in the diagram for $\Delta^k$ (which is a stack of $k$ copies of $\Delta$) the wire corresponding to $l$ ends back to its initial position.
    For any $\Delta^r$ where $r$ is not a multiple of $k$, the wire is instead not back to its initial position, so the statement follows.
\end{proof}

\begin{proposition}\label{prop:transistors:not:periodic}
    Assume that $\overline{\Delta}$ is strongly reduced.
    Then $\Delta$ is periodic if and only if it has no transistors.
\end{proposition}

\begin{proof}
    If $\Delta$ has no transistors then it is periodic by \cref{prop:period}.
    Aiming for a contradiction, let us assume that $\Delta$ has transistors and is periodic.
    Then some power $\Delta^k$ must feature a dipole $D(u,v)$.
    Therefore a sequence of base line shifts and permutations in $\overline{\Delta}$ moves the base line through $u$ or $v$, away from the wires connecting $u$ to $v$.
    But then $\overline{\Delta}$ is not strongly reduced, which is a contradiction.
\end{proof}

Let us now fix a finite semigroup presentation $\Presentation$.

\begin{algorithm}
    Let $\Delta$ be a symmetric diagram over $\Presentation$.
    \begin{enumerate}
        \item Compute $\overline{\Delta}$ and reduce it, performing base line shifts and base permutations whenever needed to unlock a dipole reduction.
        Let $\overline{\Delta'}$ be the resulting diagram.
        It is unique up to similarity by \cref{prop:dipole:reductions}.
        \item Check whether $\overline{\Delta'}$ has transistors (which only depends on the similarity class).
        It does if and only if $\Delta'$ is not periodic, by \cref{prop:transistors:not:periodic}.
        \item If $\Delta'$ is periodic, its period is the least common multiple of the winding numbers of its loops, by \cref{prop:period}.
    \end{enumerate}
Recall that $\Delta'$ and $\Delta$ have the same period by \cref{rmk:period:conjugate,,prop:conjugacy:rewriting:system}.
\end{algorithm}


\section*{Acknowledgements}

The author is thankful to Francesco Matucci and Davide Perego for fruitful conversations.

The author is supported by the ``GoFR'' project ANR-22-CE40-0004 and by the research grant PID2022-138719NA-I00 (Proyectos de Generación de Conocimiento 2022) financed by the Spanish Ministry of Science and Innovation, and is a member of the Gruppo Nazionale per le Strutture Algebriche, Geometriche e le loro Applicazioni (GNSAGA) of the Istituto Nazionale di Alta Matematica (INdAM).

No AI tools were used at any stage of production of the present paper.


\printbibliography[heading=bibintoc]

\end{document}